\documentclass[12pt]{article}

\usepackage{amsmath, amsthm, amssymb, amsfonts, enumerate}

\usepackage[longnamesfirst]{natbib}
\usepackage{dsfont}
\usepackage{calrsfs}

\usepackage[dvips]{graphicx}
\usepackage{graphics}

\newtheorem{theorem}{Theorem}

\newtheorem{proposition}[theorem]{Proposition}

\theoremstyle{definition}

\newtheorem{example}[theorem]{Example}

\numberwithin{equation}{section}

\numberwithin{theorem}{section}

\begin{document}

\title{Lowest Unique Bid Auctions}

\author{
Marco Scarsini%
   \thanks{Marco Scarsini's research was partially supported by  MIUR-COFIN. A special thank to Nicolas Katz for introducing this author to lowest unique bid auctions.}\\
   Dipartimento di Scienze Economiche e Aziendali\\
   LUISS\\
   Viale Romania 32\\
   I--00197 Roma, Italy\\
   and HEC, Paris\\
   \texttt{marco.scarsini@luiss.it}
\and
   Eilon Solan%
   \thanks{Eilon Solan's research was supported by the Israel Science Foundation (grant 212/09).}\\
   The School of Mathematical Sciences\\
   Tel Aviv University\\
   Tel Aviv 69978, Israel\\
\texttt{eilons@post.tau.ac.il},
\and
Nicolas Vieille\\
HEC Paris\\
Economics and Finance Department\\
1 rue de la Lib\'eration\\
F--78351 Jouy-en-Josas Cedex,
France\\
\texttt{vieille@hec.fr}
}

\bigskip

\date{\today}

\maketitle

\thispagestyle{empty}

\newpage

\begin{abstract}

\bigskip

We consider a class of auctions (Lowest Unique Bid Auctions) that have achieved a considerable success on the Internet. Bids are made in cents (of euro) and every bidder can bid as many numbers as she wants. The lowest unique bid wins the auction. Every bid has a fixed cost, and once a participant makes a bid, she gets to know whether her bid was unique and whether it was the lowest unique. Information is updated in real time, but every bidder sees only what's relevant to the bids she made. We show that the observed behavior in these auctions differs considerably from what theory would prescribe if all bidders were fully rational. We show that the seller makes money, which would not be the case with rational bidders, and some bidders win the auctions quite often. We describe a  possible strategy for these bidders.

\bigskip


\end{abstract}

\newpage

\section{Introduction}\label{se:intro}

The development of the Internet has seen the diffusion of many
different types of auctions, some of them with very unusual rules.
One auction method that has become common in many internet sites
is the least-unique-bid auction (LUBA). In this method, bids are
made in cents (of euro)  and each bidder can make as many bids as
she likes, paying a fixed amount for each bid. The winning bid is
the smallest positive amount that was bid by a single bidder; the
bidder who made this bid wins the auction, and pays the winning
bid. If  all positive integers
where bid by either no-one or by at least two bidders, then no-one
wins the auction. This procedure has some features that are common
with all-pay auctions and some other that are common with
lotteries.

Numerous internet sites use LUBA's to auction various objects, ranging from \$100 cash to various electronic gadgets and cars.
In these auctions the bidders can receive real-time information on their bids: they observe for each
bid whether it is currently the winning bid, whether another bidder made the same bid (and therefore it will not win),
or whether currently no-one made the same bid yet it is not the minimal bid with this property (and then it may become the winning bid).

The rules of the auction, quoted from the leading site
http://www.bidster.com/how\_it\_works, are as follows:

\bigskip
\begin{texttt}
Bidster's lowest unique bid auctions differ from regular auctions
in that the lowest and unique  bid wins the auction. Bids can only
be placed in whole pence, and the goal is to place low bids that
are unique when the auction ends. For each bid that you place,
Bidster gives you one of the following three responses:

\begin{enumerate}
\item      \textbf{Your bid is the lowest unique bid}.
If the bid maintains this status, meaning that no bidder places a
bid on the same amount, you will win the auction when it ends.

\item      \textbf{Your bid is unique, but not the lowest}.
To win the auction, all bids that are lower and unique must be
knocked-out of the auction by either you or another bidder placing
the same amount.

\item      \textbf{Your bid is unfortunately not unique}.
Amounts that receive several bids cannot win the auction.

\end{enumerate}

If your bid is the lowest unique at the end of the auction you get
the right to buy the auction item for that bid, which almost
always is a fraction of what the item would otherwise have 
cost you. Bid and win!
\end{texttt}

The motivation to use this auctioning method is that it may (and usually it does) generate more money than the value of the object that is offered.
Usually the winning price, which is the selling price, is low  compared to the value of the object, and is heavily advertized.
Yet as soon as the number of bids times the fixed cost of a bid surpasses the value of the object, the seller makes a profit, regardless of the selling price.

\bigskip

In this paper we analyze least-unique-bid auctions both from a
theoretical and a behavioral point of view.
Observe that the auction is dynamic, and the information that
bidders have evolves with time. The analysis of such a dynamic
auction is too complex, therefore we study a static version of it, where bids are made once and for all simultaneously by
all bidders.

Beyond the complexity of the dynamic setup, one additional motive for restricting ourselves
to a static setup is that  the auction sites publish
only static data (when they publish them at all).\footnote{Attempts to get
more detailed data from the auctioneers have failed.}  As a result, the empirical analysis
must focus on a static model. Nevertheless,
some of our insights extend to the dynamic auction that is used in practice.

Our static model is a traditional auction-theoretic model. We assume that each of a fixed number of bidders
values the auctioned item to its market/resale value, which is available on the site. We assume risk-neutrality and focus on Nash equilibrium analysis. Thus, the auction has known common values.

In this static model we show that the equilibrium payoff of the seller is non-positive; the argument
carries over to the dynamical setting. In the toy case where there are only two bidders,
we show that all symmetric equilibria involve
bidding intervals of numbers. By contrast, when there are at least three bidders, no such equilibrium exists.
Moreover we show with an example that,
although theoretically a seller cannot make money, a strategic bidder can actually make money in this auction.

\bigskip

There  is widespread empirical evidence that contradicts these findings. The mere fact
that  a huge number of sites  offer these
auctions suggests that it is a profitable business.\footnote{For instance the site
http://www.asteribasso.info/home.htm lists fifteen of them just in
Italy.} On the other hand, the data we use suggest that the seller gets a significantly positive return in most of the auctions.

Thus, the
hypotheses of an equilibrium analysis are not satisfactory from a descriptive viewpoint, and we here try to understand why. Data inspection suggests that
bidders adopt a variety of behaviors, going from the totally
non-strategic (not unlike a lottery player) to the entirely
strategic. The many non-strategic bidders are rationally
bounded and do not realize that, without a strategic analysis,
they are worse off than if they played a lottery. This is due to
the presence of strategic bidders, who exploit the na\"\i vet\'e
of non-strategic bidders.

The multitude of non-strategic bidders contribute to the profit of
the seller, just like in an unfair lottery. The question is now
`how can strategic bidders make money?'

A simple behavioral model shows how to reconcile
the existence of a single strategic bidder making money, together with the seller.

Notice that the recognized presence of strategic bidders helps the
seller, because it shows that the auction is not a lottery and
therefore there exist strategies that improve the chance of
winning the object. This may fool semi-strategic bidders, namely,
bidders who realize that some skill is needed to win the auction,
and wrongly assume that they have these skills or that they can
learn them by playing over and over. These bidders, even in the
best circumstances, have a learning curve, and during the learning
phase they probably bid heavily and increase the revenue of the
seller.

In this respect publishing the names of the winners is a good marketing policy for the seller. On one hand this shows that some bidders win more often than others, on the other hand they don't always win, so there is hope for a newcomer to be able to win at some point. Examining the data one notices that some names of winners are
recurring, for example in Bidster.com, among 78 auctions
conducted in December 2009, one person won 11 auctions, and 6 other bidders won together 27 auctions.

Some blog (e.g.,
http://www.pierotofy.it/pages/guide\_tutorials/Internet/La\_veri\-ta\_su\_Bidplaza\_e\_sulle\_aste\_al\_ribasso/)
suggests that strategic bidders bid only for more expensive items.
The list of winners at Bidster.com shows that this is not always
the case, since some of the recurrent winners has won also
auctions for relatively cheap items. However, data suggests that
the bidding behavior in auctions for cheap items is less strategic
than the bidding behavior in auctions for expensive items. The
intuition is that when an item is very cheap, it doesn't make
sense to bid a huge amount on it, so the typical strategy used by
strategic bidders, to bid on large intervals, is not effective.
Therefore auctions on very cheap items are more to be considered like
lotteries.

A typical strategy that is suggested in blogs is a
variation of an equilibrium strategy in the toy model with two
bidders. It involves bidding on an interval $[a, b]$, say. The
rationale for this strategy is as follows. The first goal of a
bidder is to find a relatively small free number $x$; that is, a
number that was not bid so far by any bidder. Given that such a
number is found, the second goal is to kill all the unique bids of
other bidders on numbers lower than $x$; that is, to bid on a
number that was bid by a single bidder, thereby removing it from
the list of potential winning bids. The third goal is to make it
too costly for other bidders to kill her own bid on $x$. To
achieve the second goal, $a$ should be as small as possible, and,
of course, smaller than $x$. To achieve the third goal, $b$ should
be as large as possible, and larger than $x$. A large $b$ has the
additional advantage that other possibly free numbers higher than
$x$ receive a bid, so, if some other bidder kills $x$, some other
possibility of winning the auction is preserved. Having a large
interval $[a, b]$ is costly, so that a strategic bidder faces a
trade-off between increasing the probability of winning and having
a lower bidding cost.

Based on data coming from Bidster.com, we provide some preliminary evidence that
a sophisticated strategy based on the above ideas may indeed perform well.
We proceed by introducing  bidding strategies of the following type. Irrespective of
the value of the item being auctioned, the strategy submits all integer bids between fixed percentages
of the value of the object.  We show that, had these bids been added in the past auctions that we consider,
this strategy would have achieved a limited loss. While somewhat artificial, as we discuss in Section 6, this exercise suggests that when using the dynamic setup in a clever way to limit the bidding costs,
constructing a profitable strategy may be to some extent feasible.
 Pointing at such a winning
strategy is outside the scope of our paper.

\bigskip

Many of the sites that offer LUBA's now offer also penny auctions,
as studied by \citet{Aug:mimeo2009, PlaPriTap:mimeo2009,
Hin:mimeo2010}. The auction mechanism of LUBA's and of penny
auctions has some analogies with all-pay auctions as studied by,
e.g.,
\citet{HenWeiWil:IER1988, BayKovVri:AER1993, Sie:E2009}.

For the classical theory of standard auctions the reader is referred, e.g., to \citet{Kle:Princeton2004, Mil:Cambridge2004, Kri:Academic2009}.

Several papers have examined empirical behavior of bidders in online auctions
(\citealp[see, e.g.,][]{RedDas:SS2006, ShmRusJan:AAP2007} \citealp[and the collection][]{JanShm:Wiley2008}).

Although lowest-unique-bid auctions are a relatively new
phenomenon, there already exist a literature that tries to explain
them.

\citet{OesWanChoCam:mimeo2008, HouvdLVel:mimeo2008, RavVir:IJIO2009,  RapOtsKimSte:mimeo2009} consider different models with the common feature that all bidders are restricted to submit at most one bid. Their results have theoretical interest, but are clearly far away from providing an explanation of the real life development of these auctions.

\citet{EichVin:mimeo2008} allow the possibility of multiple bids, but restrict severely the class of possible strategies for bidders. In their model the number of bidders is known and each bidder can bid only on an interval of numbers starting from zero. They find the symmetric Nash equilibrium for this game and study several data set of German auctions to see how far real life is from their equilibrium model.

\citet{Gal:mimeo2009} studies a dynamic model in discrete time, where bidders receive signals that reflect the behavior of other bidders. He models explicitly the amount of bids that each bidder makes  in a game with an auctioneer and a known number of potential bidders, and shows that, when the agents are rational, the profit of the auctioneer can be positive only if her valuation of the good is much lower than the valuation of the bidders. Then he considers a model where the bidders have bounded rationality.

The idea of a strategic bidder among a multitude of noisy bidders traces back to some models in finance \citep[see, e.g.,][]{Kyl:E1985, Kyl:RES1989}.

The paper is organized as follows.
In Section~\ref{se:model} we present the static game that corresponds to this auction method,
and in Section~\ref{se:strategic} we present an equilibrium analysis of this game.
In Section~\ref{se:behavioral} we propose a behavioral model that better describes
the observed behavior of bidders. In Section~\ref{se:data} we analyze some data of real LUBA auctions. Section~\ref{se:comments} provides concluding remarks.

\section{The setup}\label{se:model}

Bidders participate in an auction for an object whose common value is $v$. We denote by $\mathcal{B} :=\{1, \dots, n\}$ the set of bidders.
Each bidder submits a (finite) set of natural numbers, which are termed her bids;
each bid costs a fixed amount   $c$,
so if a bidder submits a set of $k$ natural numbers, then she pays $kc$.
The bidders' decisions  are made simultaneously and independently.
We denote the set that contains all the bids of bidder $i$ by $S_i$.
The set of all positive integers that were chosen by only one bidder is denoted by $S_*$:
\[
S_* = \{s \in \mathbb{N} \colon \#\{i \in \mathcal{B} \colon s \in S_i\} = 1\}.
\]
If $S_*$ is empty then there is no winner.
If $S_*$ is non-empty then the winning bid is
\[
s_* = \min(S_*),
\]
and the bidder who made it, $i_*$, is termed the ``winner.'' She receives the object and pays the price $s_*$ for it.

Therefore the payoff to each bidder $i$ is
\[
u_i(S_1,S_2,\ldots,S_n) =
\begin{cases}
- c |S_i| +v - s_* &  \text{ if $i = i_*$,}\\
- c |S_i| & \text{ if $i \neq i_*$,}
\end{cases}
\]
where $|S_i|$ is the cardinality of the set $S_i$.

This defines a symmetric game among the bidders. For equilibrium
analysis, we may assume that the set of strategies of each bidder
is bounded. Indeed, in an equilibrium no bidder will bid more than
$v$ on the item, and the number of bids of no bidder will exceed
$v/c$. It is well known that a finite symmetric game
always admits a symmetric equilibrium, i.e., an equilibrium in
which all the bidders play the same strategy, and receive the same
expected payoff. Unless the game is degenerate, the equilibrium
will be in mixed strategies. Usually there are also asymmetric
equilibria.

The payoff to the seller is $c$ times the total number of bids,
plus the winning bid, minus the value of the object:
\[
u_S = \sum_{i=1}^n c |S_i| + s_* - v.
\]
The game between the bidders and the seller is
a zero sum game, since the sum of the payoff to the bidders
is minus the payoff to the seller.

\section{The strategic analysis}\label{se:strategic}

A {\em pure strategy} of a bidder is a choice of a subset of
natural numbers $\mathbb{N}$, and, as mentioned above, we can
restrict each bidder to a finite subset of pure strategies. As
usual, a {\em mixed strategy} is a probability distribution over
pure strategies. A mixed strategy is called {\em monotone} if all
the subsets in its support are of the form $\{1,2,\ldots,k\}$, for
some $k \in
\mathbb{N}$.

In the next subsection we analyze the game with two bidders.
Plainly this analysis does not arise in practice, but it
highlights the role of monotone strategies in equilibrium. In the
subsequent subsection we show that when there are more than two
bidders, there are no equilibria in monotone strategies, so in
particular the computational complexity of equilibrium strategies
is huge.

\subsection{Two bidders}

In this subsection we study a toy model with only two bidders.
Our first result shows that in this case {\em all} equilibria are in monotone strategies.
We then identify the unique symmetric equilibrium.

\begin{theorem}\label{th:monotone}
When $k=2$, all equilibria are in monotone strategies.
\end{theorem}

\begin{theorem}\label{th:strategy}
When $k=2$, there is a unique symmetric equilibrium $(\sigma,\sigma)$.
Let $N$ be the maximal natural number that satisfies
\begin{equation}\label{eq:equilstrat2play}
\frac{c}{v-1} + \frac{c}{v-2} + \cdots + \frac{c}{v-N} < 1.
\end{equation}
The strategy $\sigma$ is defined by
\begin{enumerate}[{\rm (i)}]
\item
$\sigma(\varnothing)=\displaystyle \frac{c}{v-1}$,
\item
$\sigma(\{1,\cdots,l\})=\displaystyle\frac{c}{v-(l+1)}$ for $l=1,2,\ldots, N-1$,
\item
$\sigma(\{1,\cdots,N\})=1-\displaystyle\sum_{l=0}^{N-1}\frac{c}{v-(l+1)}$.
\end{enumerate}
\end{theorem}

\begin{proof}[Proof of Theorem~\ref{th:monotone}]
Let $(\sigma_1,\sigma_2)$ be an arbitrary equilibrium. We argue by
contradiction and assume that, say, $\sigma_1$ is not a monotone
strategy. That is, there is a subset $\bar{S}_1\subset \mathbb{N}$
not of the form $\{1,2,\ldots,k\}$, such that $\sigma(\bar{S}_{1}) >
0$. It follows that there is an integer $l\geq 1$ such that $l\notin
\bar{S}_1$ but $l+1\in \bar{S}_1$.

Denote by $\bar{S}'_1$ the set obtained by replacing $l+1$ with $l$
in $\bar{S}_1$, that is,
$\bar{S}'_1=\left(\bar{S}_1\setminus\{l+1\}\right)\cup\{l\}$.
Denote by $\bar{S}''_1 = \bar{S}_1 \cap \{1,2,\ldots,l-1\}$
the set obtained by removing from $\bar{S}_1$ all bids higher than $l-1$.
Plainly,
\[ |\bar{S}''_1| < |\bar{S}_1| = |\bar{S}'_1|. \]

Given two subsets $S_1,S_2\subset \mathbb{N}$,
we say that $S_1$  \emph{wins} against $S_2$, if bidder 1 wins when she bids $S_1$ and her opponent bids $S_2$.
The following table lists
three possible cases that may occur when
each of the three sets $\bar{S}_1$, $\bar{S}'_1$ and $\bar{S}''_1$ is used by bidder 1,
and the bid of bidder 2 is some set $S_2$.

\[
\begin{array}{l||l|l|}
\bar{S}_1 & \bar{S}'_1 & \bar{S}''_1\\
\hline
\hbox{does not win} & \hbox{may win or lose} & \hbox{does not win}\\
\hbox{wins, winning bid} \geq l+1 & \hbox{wins, winning bid} = l& \hbox{does not win}\\
\hbox{wins, winning bid} < l & \hbox{wins, winning bid} < l & \hbox{wins, winning bid} < l
\end{array}
\]

\noindent
Indeed, if $\bar{S}_1$ does not win against $S_2$ then $\bar{S}''_1$,
which is a subset of $\bar{S}_1$, does not win against $S_2$, either.
However, $\bar{S}'_1$ may win against $S_2$, in case $S_2$ does not bid on $l$,
and the winning bid when the bidders use $\bar{S}_1$ and $S_2$ is higher than $l$.
If $\bar{S}_1$ wins against $S_2$ and the winning bid is at least $l+1$,
then as above $\bar{S}''_1$ does not win against $S_2$.
But in this case $\bar{S}'_1$ wins against $S_2$:
since the winning bid is at least $l+1$ and  $\bar{S}_1$ does not bid on $l$,
it follows that $l \not\in S_2$,
so that $l$ is the winning bid when the bidders choose $\bar{S}'_1$ and $S_2$.
Finally, if $\bar{S}_1$ wins against $S_2$ and the winning bid is smaller than $l$,
then since
\[
\bar{S}_1 \cap \{1,2,\ldots,l-1\} = \bar{S}'_1 \cap \{1,2,\ldots,l-1\} = \bar{S}''_1 \cap \{1,2,\ldots,l-1\},
\]
it follows that the outcome is the same,
whether bidder 1 uses $\bar{S}_1$, $\bar{S}'_1$ or $\bar{S}''_1$.

As we see, $\bar{S}'_1$ does at least as good as $\bar{S}_1$.
In fact, if there is some $S_2$ that is chosen with positive probability by $\sigma_2$,
such that $\bar{S}_1$ wins against $S_2$ and the winning bid is at least $l+1$,
then $\bar{S}'_1$ does strictly better than $\bar{S}_1$,
contradicting the equilibrium property of $\sigma_1$.

If, on the other hand, for every $S_2$ that is chosen with positive probability by $\sigma_2$,
whenever $\bar{S}_1$ wins against $S_2$ the winning bid is less than $l$,
than $\bar{S}''_1$ does at least as good as $\bar{S}_1$
but it contains less bids, and therefore it achieves higher profit than $\bar{S}_1$,
again, contradicting the equilibrium property of $\sigma_1$.

This implies that the assumption is incorrect: there is no such subset $\bar{S}_1$ that is played with positive
probability under $\sigma_1$. In other words, $\sigma_1$ is monotone.
\end{proof}

\begin{proof}[Proof of Theorem~\ref{th:strategy}]
Let $(\sigma,\sigma)$ be a symmetric equilibrium. By Theorem~\ref{th:monotone}, $\sigma$ is a monotone strategy.
We proceed in three steps.

\medskip

\noindent
\textbf{Step 1}. Let $S=\{1,\ldots, l\}$ be a non-empty set such that $\sigma(S)>0$, and define $S':=\{1,\ldots, l-1\}$.
We claim
that $\sigma(S')>0$.

Indeed, otherwise the other bidder bids on $l-1$ only when she bids
on $l$.
Therefore when facing $\sigma$, bidding $S'$ would win exactly
when $S$ would win. Since $|S'|<|S|$, the expected payoff of
bidder 1 is higher when facing $\sigma$ and bidding $S'$ rather
than $S$: a contradiction.

\medskip

We deduce that if $\sigma$ does not assign probability 1 to  not participating,
that is, if $\sigma(\varnothing) < 1$,
then $\sigma$ assigns a positive probability to some set of the form $\{1,2,\ldots,l\}$,
and so by an iterative use of step 1,
$\sigma$ assigns a positive probability to the set $\{1\}$.

\medskip

\noindent
\textbf{Step 2}. $\sigma(\varnothing)>0$. In particular, the expected payoff of each bidder is zero.

Because the other bidder uses $\sigma$, which is a monotone strategy,
whenever the other bidder participates she bids on $1$.
As mentioned above, $\sigma$ assigns a positive probability to bidding $\{1\}$.
But the bid $\{1\}$ is winning only if the other bidder does not participate.
Because the cost of this bid is $c$, and the bidder can receive at least 0 by not participating, she must win the auction with positive probability when she bids $\{1\}$. This implies that $\sigma$ assigns a positive probability to not participating, as claimed.

\bigskip

\noindent
\textbf{Step 3}. (\emph{Characterization of $\sigma$})
Let $N$ be the largest integer such that $\sigma(\{1,\ldots,
N\})>0$. By \textbf{Steps 1} and \textbf{2}, when facing $\sigma$,
the expected payoff of a bidder when bidding $\{1,\dots, l\}$ is
equal to zero, for each $l=1,\ldots, N$. The cost of such a bid is
equal to $nc$. On the other hand, the expected gain can be
computed as follows. With probability $\sigma(\varnothing)$, the
opponent bidder will not post any bid, the winning bid is 1, and
the gain is thus $v-1$; for $i<l$, the probability that the
opponent will bid $ \{1,\ldots, i\}$ is $\sigma(\{1,\ldots, i\})$.
In that case, the winning bid is $i+1$, and the gain, equal to
$v-(i+1)$. The zero payoff condition thus yields
\begin{equation}\label{eq:zeropayoff}
lc=\sigma(\varnothing)(v-1)+\sum_{i=1}^{l-1}\sigma(\{1,\ldots, i\}) (v-(i+1)), \quad\text{for all}\ l \in \{1,2,\ldots,N\}.
\end{equation}
For $l=1$ we obtain
\begin{equation*}
\sigma(\varnothing)=\frac{c}{v-1}.
\end{equation*}
When subtracting the expression in \eqref{eq:zeropayoff} for $l$ from the one for $l+1$, one gets
\begin{equation*}
\sigma(\{1,\ldots, l\}=\frac{c}{v-(l+1)}\quad \text{for each }l=1,\ldots, N-1.
\end{equation*}
Because $\sigma$ is a probability distribution we deduce that
\[
\sigma(\{1,\ldots, N\})=1- \sum_{i=1}^{N}\frac{c}{v-i}.
\]
In particular, this implies that $N$ is such that
\[
1- \sum_{i=1}^{N}\frac{c}{v-i}< 1.
\]
On the other hand, by the equilibrium condition, the payoff of a bidder,
when facing $\sigma$ and bidding $\{1,\ldots, N+1\}$, is non-positive.
Arguing as above, this yields
\[
1- \sum_{i=1}^{N+1}\frac{c}{v-i}\geq 1.
\]
This uniquely characterizes $N$ as the largest integer such that
\[
1- \sum_{i=1}^{N}\frac{c}{v-i}< 1,
\]
and concludes the proof of Theorem~\ref{th:strategy}.
\end{proof}

\subsection{Three bidders or more}

The following theorem,
which states that when there are at least three bidders there is no symmetric equilibrium in monotone strategies,
stands in sharp contrast to the results in the previous section.

\begin{theorem}\label{th:threebidders}
If the number of bidders is at least three, and if $v>\max\{2c+2,10\}$, then there is no symmetric equilibrium in
monotone strategies.
\end{theorem}

The condition $v>\max\{2c+2,10\}$ means that the value of the object is not too low, 
and is always satisfied in practice.
Theorem~\ref{th:threebidders} suggests that when there are at least three bidders
the equilibrium is significantly more complex than the equilibria for two bidders.

\begin{proof}[Proof of Theorem~\ref{th:threebidders}]
Let the number of bidders be $k\geq 3$. Arguing by contradiction, we assume that there is an equilibrium $(\sigma,\cdots, \sigma)$, where $\sigma$ is a monotone strategy.

We  follow the proof of Theorem~\ref{th:strategy}. Observe that \textbf{Steps 1} and \textbf{2}
in that proof are valid irrespective of the number of bidders.
In particular, from Step 2 the expected payoff of each bidder is 0.
We now adapt the computation in \textbf{Step 3}.

Let $N$ be the largest integer such that $\sigma(\{1,\ldots, N\})>0$.
Since $v>2c+2$, one has $N\geq 2$.
Otherwise $\sigma$ bids only on the empty set $\varnothing$ or on the set $\{1\}$.
But then, when all other bidders use $\sigma$,
the bid $\{1,2\}$ wins with probability 1, and
yields a positive expected payoff of at least $v-2c-2$.
This contradicts the equilibrium condition.

For the sake of readability, we set $x_\varnothing=\sigma(\varnothing)$ and $x_i=\sigma(\{1,\ldots,i\})$, for $i=1,\ldots, N$.
The proof is organized as follows. First we write the zero payoff condition for the two bids $\{1\}$ and $\{1,2\}$. We next show that the bid $\{2\}$ then yields a positive expected payoff when facing $(\sigma,\ldots, \sigma)$, and is therefore a profitable deviation.

The bid $\{1\}$ costs $c$ and is winning if and only if all other
$k-1$ bidders fail to submit a bid -- an event that occurs with
probability $x_\varnothing^{k-1}$. In that case, the winning bid is
equal to 1. Thus,
\begin{equation}\label{eq:payoff1}
c=(v-1)x_\varnothing^{k-1}.
\end{equation}
The bid $\{1,2\}$ costs $2c$ and is winning if either all other
$k-1$ bidders fail to submit a bid, or if at least one of them
submits a non-empty bid, and all of those who submit a non-empty
bid actually submit the bid $\{1\}$. The probability of the former
event is $x_\varnothing^{k-1}$, and the probability of the latter
one is $(x_\varnothing +x_1)^{k-1}-x_\varnothing^{k-1}$. The winning
bid is equal to 1 in the former case, and it is equal to 2 in the
latter. Thus, the zero payoff condition of the bids $\{1,2\}$
becomes
\begin{equation}\label{eq:payoff12}
2c= (v-1)x_\varnothing^{k-1}+(v-2) \left((x_\varnothing +x_1)^{k-1}-x_\varnothing^{k-1}\right).
\end{equation}
Given \eqref{eq:payoff1}, it is not difficult to see that \eqref{eq:payoff12} is equivalent to
\begin{equation}\label{eq:payoff12bis}
c= (v-2) \left((x_\varnothing +x_1)^{k-1}-x_\varnothing^{k-1}\right)=(v-2)\left((x_\varnothing +x_1)^{k-1} -\frac{c}{v-1}\right).
\end{equation}
From \eqref{eq:payoff1} and \eqref{eq:payoff12bis} one therefore deduces
\begin{equation}\label{eq:proba12}
x_\varnothing= \sqrt[k-1]{\frac{c}{v-1}}\quad \text{and}\quad x_1=\sqrt[k-1]{\frac{c}{v-2}+\frac{c}{v-1}}-\sqrt[k-1]{\frac{c}{v-1}}.
\end{equation}
We now consider the bid $\{2\}$. When facing $(\sigma,\ldots,
\sigma)$, this bid is winning if each of the other bidders either
submits an empty bid or the bid $\{1\}$, and the number of
bidders who submit the bid $\{1\}$ is \emph{not} equal to 1 (an
event that occurs with probability
$(x_\varnothing+x_1)^{k-1}-(k-1)x_1 x_\varnothing^{k-2}$), and the
winning bid is then equal to 2. By the equilibrium condition, this
bid does not yield a positive payoff. Thus, a necessary
equilibrium condition is
\[c \geq (v-2)\left((x_\varnothing +x_1)^{k-1}-(k-1)x_1 x_\varnothing^{k-2}\right),\]
which, in light of the previous equalities, can be written as
\[(k-1)x_1x_\varnothing^{k-2}\geq \frac{c}{v-1},\]
or equivalently, by \eqref{eq:payoff1},
\begin{equation}\label{eq:neccond}
(k-1)x_1\geq x_\varnothing.\end{equation}
Since
\[
\frac{c}{v-1}\leq \frac{c}{v-2},
\]
\eqref{eq:proba12} and \eqref{eq:neccond} imply
\[(k-1)\sqrt[k-1]{\frac{2c}{v-2}}\geq k\sqrt[k-1]{\frac{c}{v-1}},\]
which is equivalent to
\begin{equation}\label{eq:neccond2}
2\left(1+\frac{1}{v-2}\right)\geq
\left(1+\frac{1}{k-1}\right)^{k-1}.\end{equation} Since $v>10$,
the left-hand side of \eqref{eq:neccond2} is less than $9/4$. On
the other hand, the right-hand side of \eqref{eq:neccond2} is
increasing in $k$ and is equal to $9/4$ for $k=3$. Hence
(\ref{eq:neccond2}) cannot possibly hold. This proves the desired
contradiction.
\end{proof}

\subsection{Equilibrium payoffs}

In this subsection we discuss the equilibrium payoffs when the number of bidders is arbitrary.
Our first observation is that the seller cannot make
money.

\begin{theorem}
\label{th:nomoney}
In all equilibria, the expected payoff to the seller is non-positive.
\end{theorem}

\begin{proof}
Consider the zero-sum game with $n$ bidders and one seller.
Each bidder $i$ has the option not to participate in the auction (set $S_i = \varnothing$)
and obtain 0.
Therefore in all equilibria, the expected payoff of each bidder is non-negative.
Since the game is zero-sum, it follows that the expected payoff of the seller is non-positive.
\end{proof}

This result carries over to the dynamic setup, in which information is revealed to the bidders:
the same argument shows that in this more elaborated game,
in all equilibria the expected payoff of each bidder is non-negative,
and therefore the expected payoff of the seller is non-positive.
It is interesting to contrast this result with real data:
according to the information provided by, e.g., Bidster.com,
the total cost of the bids far exceeds the value of the object,
and the expected payoff to the seller is positive (see Section~\ref{se:data} for a more elaborate discussion on this issue).
This shows that at present the participants in the auctions do not play a Nash equilibrium.
A possible explanation to this phenomenon may be that the bidders' utility is different
than the one presented here: e.g., they may be risk averse/lovers,
or they may have positive utility from participating in the game.

A natural question is whether the bidders can profit from participating.
Interestingly, if the number of participants is high,
then the expected payoff of at least one bidder is 0.
This implies in particular that in any symmetric equilibrium the expected payoff of all bidders is 0.

\begin{proposition}\label{pr:prop1}
If the number of bidders $n$ exceeds $v/c$,
then in any equilibrium, the expected payoff of at least one bidder is  $0$.
\end{proposition}

\begin{proof}
If there is a positive probability that some bidder, say bidder $i$,
will not participate (that is, she will set $S_i = \varnothing$),
then, as mentioned above, her expected payoff is 0.

Otherwise, all the bidders participate with probability 1.
This implies that the total number of bids exceeds $v/c$,
but this would make the expected payoff to the seller positive,
which is ruled out by Theorem~\ref{th:nomoney}.
\end{proof}

The argument in the proof of Proposition~\ref{pr:prop1} implies that
in every equilibrium, if the expected number of bidders who participate exceeds $v/c$
then the expected payoff of all bidders is 0.
Surprisingly, if the expected number of bids is smaller than $v/c$,
then some bidders can make a profit.
This is illustrated by the next example.

\begin{example}
Suppose that there are two bidders $n=2$, the bidding cost is $c=1$, and the value of the object is $v=4$.

Bidder 1 always bids $1$, whereas bidder 2 uses the following strategy:
with probability $\frac{1}{2}$ she bids $\{1,2\}$ and with probability $\frac{1}{2}$ she does not participate
(she bids $\varnothing$).

Let us verify that this is an equilibrium.
If bidder 2 bids $\{1,2\}$ she wins the object, and her payoff is $-2+4-2 = 0$.
Therefore she is indifferent between bidding $\{1,2\}$ and not participating.
Bidder 2 cannot profit by deviating: if she bids $\{1\}$  or $\{2\}$ she surely loses,
and  bidding $\{3\}$ would be unprofitable for her.

Bidder 1 wins only if bidder 2 does not participate, which happens with probability $\frac{1}{2}$.
Hence her expected payoff is $-1 + \frac{1}{2}(4-1) = \frac{1}{2}$.
That is, bidder 1 has a positive profit.

Bidder 1 cannot profit by not participating, because by not participating her payoff decreases to 0.
We now verify that bidder 1 cannot profit by bidding on a set $S$, different than $\{1\}$.
Indeed, suppose that $S$ is not empty and is different than $\{1\}$.
If bidder 2 does not participate, bidder 1 wins and her profit is less than $v-c-1$,
which is her profit when she bids $\{1\}$.
If on the other hand bidder 2 bids $\{1,2\}$,
to increase her payoff bidder 1 must win, and therefore must bid on at least 3 numbers:
she must bid on $1$ and $2$ so that bidder 2 does not win, and on another number to ensure that she herself wins.
But then her profit is at most $v-3c-3$, which is negative.

The expected number of bids is $2$, and the expected winning bid is $\frac{3}{2}$.
Therefore the seller's profit is $2+\frac{3}{2}-4=-\frac{1}{2}$, which is negative.

This example can be generalized to any number of bidders:
when the number of bidders is $n$,
there is an equilibrium in which bidder 1 always bids $\{1\}$ (and makes a positive expected payoff),
and all the other bidders either bid $\{1,2\}$ or do not participate.

The example could be enriched  by adding non-strategic bidders to the  two strategic bidders. Even in this case one of the strategic bidders would have a positive profit, and the other would not lose money. 
\end{example}

Though it seems that the winning bid is low,
when the number of bidders is high,
in equilibrium there will be high bids.

\begin{proposition}
If the number of bidders $n$ exceeds $v/c$
then in every equilibrium there will be a bid of at least $v/(c+1) - 1$.
\end{proposition}

\begin{proof}
Fix an equilibrium in the game.
Let $k$ be the maximal integer which some bidder bids with positive probability in equilibrium.
By Proposition~\ref{pr:prop1}, there is at least one bidder $i$ whose equilibrium payoff is 0.
If bidder $i$ bids on the $k+1$ numbers $1,2,\ldots,k+1$ then she surely wins, and her payoff is
$-c(k+1) + v - (k+1)$.
Since we are in equilibrium,
bidder $i$ cannot profit by bidding on all the numbers $1,2,\ldots,k+1$,
which implies that $-c(k+1) + v - (k+1) \leq 0$, so that
\[
k \geq \frac{v}{c+1}-1,
\]
as desired.
\end{proof}

\section{A behavioral model}\label{se:behavioral}

At first sight, the  empirical evidence may appear puzzling. On the one hand, the seller makes money. Since the game is a zero-sum game between the bidders and the seller, bidders on average lose money. Yet, the presence of recurring names in the list of winners suggests that sophisticated bidders may make money.  We here exhibit a simple behavioral model that reconciles these observations. It is directly inspired by the literature on financial market microstructure.

Assume that there is a unique strategic bidder, who faces a random number of ``noisy" bidders, who bid in a random way.

We denote by $N_i$ the random number of bidders who bid on $i$ and we assume that for each $i$,
the random variable $N_i$ follows a Poisson distribution with parameter $\lambda_i$, and that all variables $N_i$ are independent.
To model the fact that the number of bids on $i$ tends to decrease with $i$, we assume that $\lambda_i$ is non-decreasing.
One specification is $\lambda_i=v/c i^\alpha$, for each $i$, and some $\alpha>0$. With this parametric model, $\mathbb{E}[N_i]=\lambda^i$, and
the expected revenue to the seller is
\[2v\sum_i\frac{1}{i^\alpha},\]
which exceeds the value $v$ of the object provided $\alpha$ is not too large -- in particular if $\alpha=2$.

Note also that the expected number of (noisy) bids is linear in $v$. Empirically, it seems that the number of bids is indeed increasing in $v$.
Whether the linear specification is reasonable is beyond the scope of this note.

Assume that $\alpha=2$ and consider the ``brute force'' strategy
of the strategic bidder that consists in bidding all $i$ up to
some $b$. We write $b=\gamma\sqrt{v/c}$, and  optimize over
$\gamma$.

The cost of such a strategy is $c\times \gamma\sqrt{v/c}=\gamma\sqrt{cv}$.
The probability of winning the auction is at least the probability
that $N_b=0$, since the former is a sufficient (but not necessary)
condition for winning. Hence, using  the Poisson specification we
obtain that the probability of winning is at least $\exp(-1/\gamma^2)$.

Thus, the expected gain is at least
\[ve^{-1/\gamma^2}-\gamma\sqrt{cv}.\]

Taking the derivative with respect to $\gamma$ to optimize we see that the optimal $\gamma$ solves
\[
\frac{2}{\gamma^3}e^{-1/\gamma^2}=\sqrt{\frac{c}{v}}.
\]

When, e.g., the value of the object is $v=500$,
and the cost of each bid is $c=0.5$, we obtain $v/c=1,000$, and $\gamma$ roughly solves
\[\frac{1}{\gamma^3}e^{-1/\gamma^2}=\frac{1}{60}.\]
Numerical evidence suggests $\gamma=4$. The expected gain is then
roughly equal to $0.87v$,  which corresponds to a huge gain.

Thus, when there is a single strategic bidder,
it is very easy to come up  with a strategy that ensures a high gain for that bidder, and a positive gain for the seller as well.
As we will see in the next section, analyzing the data suggests that there might be room for making money
even in the presence of several strategic bidders.

\section{Analysis of Data}
\label{se:data}

We collected the complete data on all Bidster.com auctions that were performed between June
2009 and January 2010. For homogeneity purposes, we limited ourselves to  items whose published
value is between 450 Euro and 1300 Euro. These items include fashionable
consumer goods, like Apple products and Nokia devices, but also  500 Euro coupons
that can be used for bidding in other actions.\footnote{Our dataset is available upon request.}
For each auction in the sample, we obtained the type and (published) market value of the item being sold, the bidding cost, the winner's name and the entire distribution of bids.
 We thus obtained data on 210 auctions, with 170 auctions on `real' items and 40 on coupons. During the period under study, the cost of bidding was switched from 2 euros to 0.50 euro.\footnote{This change seems to have taken place over a few days, with more and more auctions involving a low bidding cost.}  Out of the 210 auctions, 113 were carried with a low bidding cost (13 of them were coupon auctions), 
 and 97 were carried with a high bidding cost.%
\footnote{Out of these auction data, one auction was carried with a bidding cost equal to \textit{zero}, probably for advertising purposes, and two auction data were incomplete, on the claimed ground that bids were too numerous. These three auctions were deleted from the sample.}

As we already mentioned, data analysis reveals that unlike the theoretical prediction
(Theorem \ref{th:nomoney}), the seller makes money. The origin of this
phenomenon may be (1) the participation of non-strategic bidders,
and/or (2) the fact that the utility of the bidders is not their
expected profit, but the bidders also derive utility from the
participation itself. In that sense, LUBA is not different from
lottery or day trading by non professionals.

As mentioned in the introduction, the list of winners include
recurring names. This suggests that some of the bidders have good
strategies that make money. Data suggests that some bidders use a
``block strategy'': a strategy that bids on a block of numbers.
As an illustration, we provide in the next figure the number of bids submitted on each integer between 560 and 680 cents, in the auction \#7300. We emphasize that such a pattern is representative of what happens throughout our sample

\bigskip

\begin{center}
FIGURE~\ref{fi:Numberofbids} ABOUT HERE
\end{center}¥

\bigskip

\begin{figure}[b]
\includegraphics[width=0.9\textwidth]{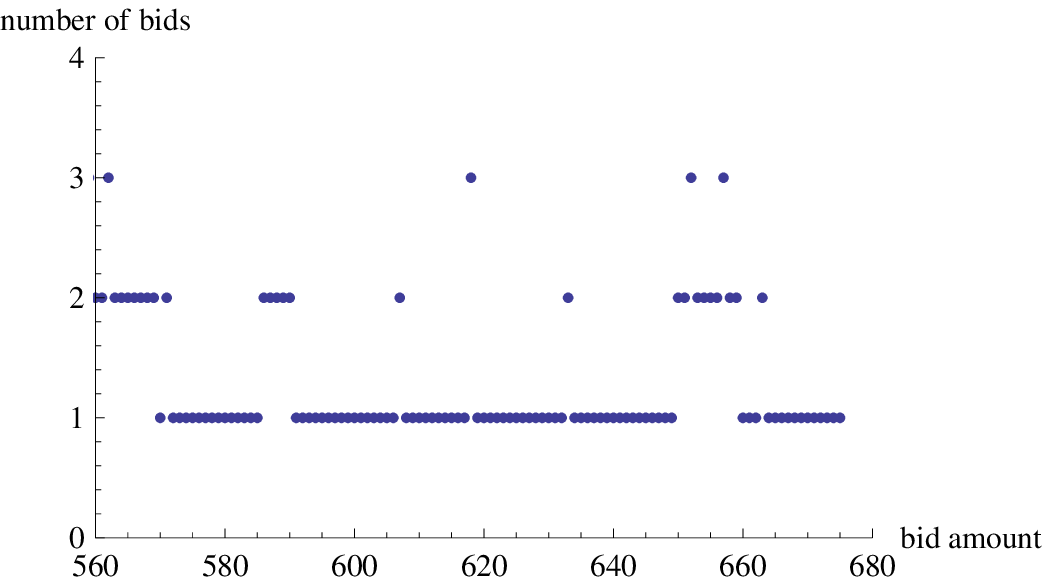}
\caption{Number of bids on each integer in auction \#7300.}
\label{fi:Numberofbids}
\end{figure}

Figure~\ref{fi:Numberofbids} displays the number of bids made on the integers from 560 to 680.
As is clear from this figure, all these integers  were offered by at least one bidder, and the overwhelming majority of these were submitted only once. This strongly suggests that a single bidder submitted all bids in this interval, while other bids were submitted by much less active bidders.

We also recall that the use of such a
block strategy has  been  advocated on the Internet (e.g., in the blog already quoted).

In Subsection~\ref{suse:data seller}, we report basic observations on the seller's revenue, and on the distribution of bids. In Subsection~\ref{suse:data bidder} we investigate the performance of block strategies.
We find that a significant number of such strategies  lose roughly 10\% to 15\% of the amount spent on it.
This finding suggests that a more sophisticated strategy,
that uses the dynamic data available to the bidders during the auction may in fact profit.
In Subsection~\ref{suse:winner} we elaborate on a possible strategy of the winner.

To allow for meaningful comparisons, and unless otherwise specified, we include here the auctions on true items, that were conducted with a bidding cost of 0.5.

\subsection{The Seller}\label{suse:data seller}

The graph in Figure~\ref{fi:Winningvaluation} provides the winning bid as a function of
the value of the sold item.  Both the value of the object and the winning bid are expressed in euros. As can be observed, the winning bid
hardly ever exceeds 2\% of the value of the object, so that the
contribution of the winning bid to the profit of the seller is
insignificant: as we mentioned earlier, the winning bid serves as
a device to attract potential participants, and not as a device to
increase the seller's profit.

\bigskip

\begin{center}
FIGURE~\ref{fi:Winningvaluation} ABOUT HERE
\end{center}¥

\bigskip

\begin{figure}[b]
\includegraphics[width=0.9\textwidth]{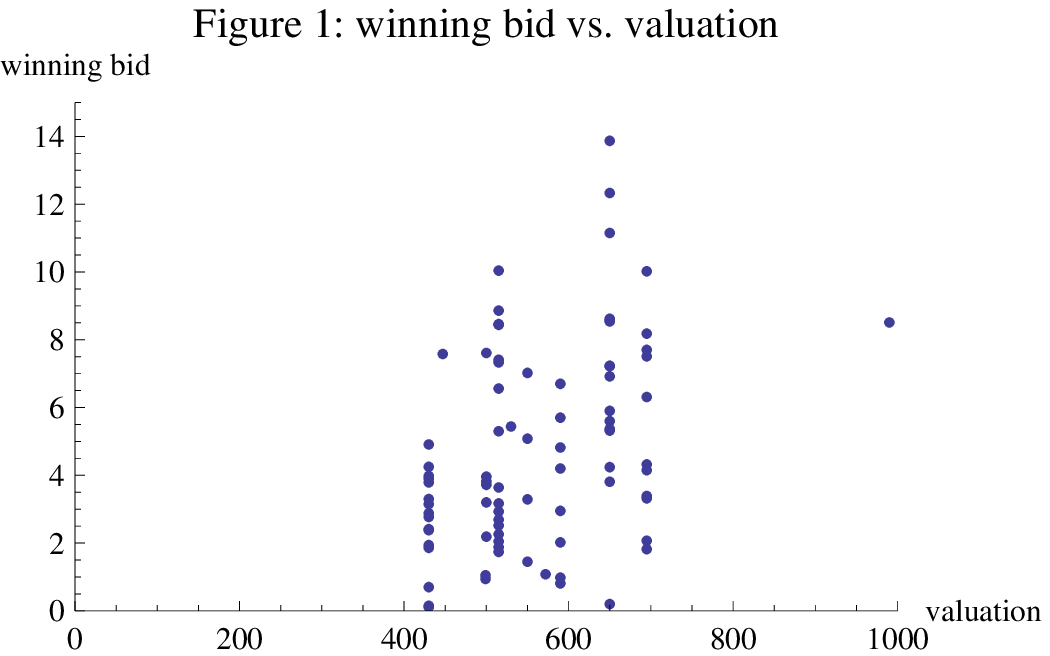}
\caption{Winning bid vs. valuation}
\label{fi:Winningvaluation}
\end{figure}

The graph in Figure~\ref{fi:Profitvaluation} provides the sellers's profit as a function of the market value of the item. Amounts are expressed in euros.

\bigskip

\begin{center}
FIGURE~\ref{fi:Profitvaluation} ABOUT HERE
\end{center}¥

\bigskip

\begin{figure}[b]
\includegraphics[width=0.9\textwidth]{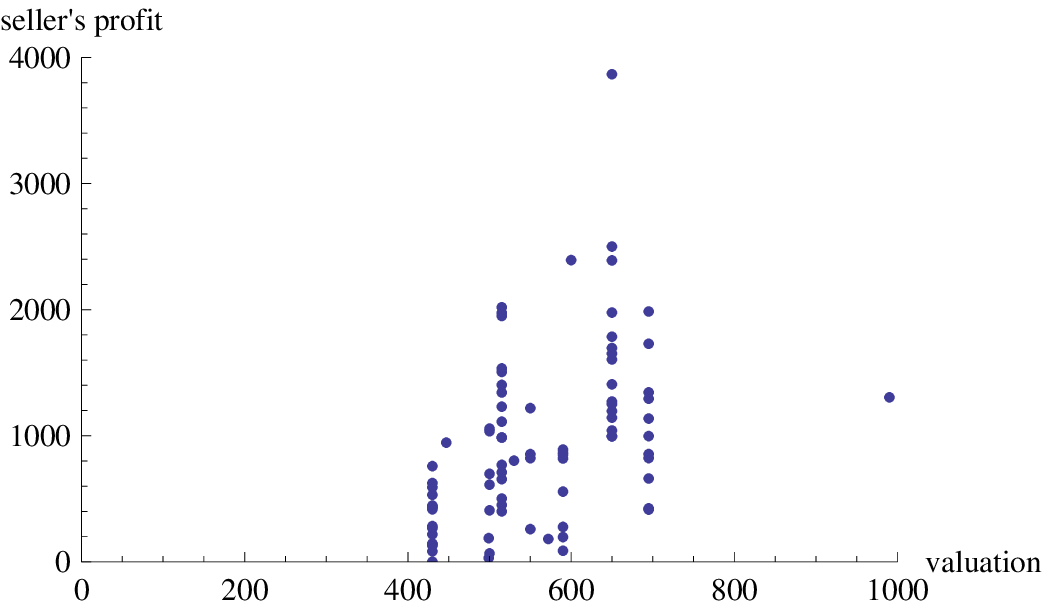}
\caption{Seller's profit vs. valuation}
\label{fi:Profitvaluation}
\end{figure}

Thus, in all auctions the seller's profit is positive, and most often significantly so. In fact, the average profit of the seller is 1121.32 euros, with a
standard deviation of 827.4.

As mentioned before, this observation
implies in particular that the participants do not play a Nash equilibrium, assuming
their utility is only derived from their monetary profit/loss.
Note that because the seller's profit is calculated w.r.t. the
published value of the item, rather than w.r.t. to price the
seller actually pays for the object, the seller's profit is in
fact higher than the figure provided above.

The profit of the seller on auctions for coupons is significantly below the profit made on auctions for true items. It seems that the goal of these auctions is to induce
participants to bid in other auctions, and thereby to provide an
indirect profit to the seller.

\subsection{The Bidders}
\label{suse:data bidder}

The graph in Figure~\ref{fi:CDF} provides the CDF of the percentage of the
winning bid from the object's value. It includes only the auctions which were conducted with a bidding cost of $0.5$.

\bigskip

\begin{center}
FIGURE~\ref{fi:CDF} ABOUT HERE
\end{center}¥

\bigskip

\begin{figure}[b]
\includegraphics[width=0.9\textwidth]{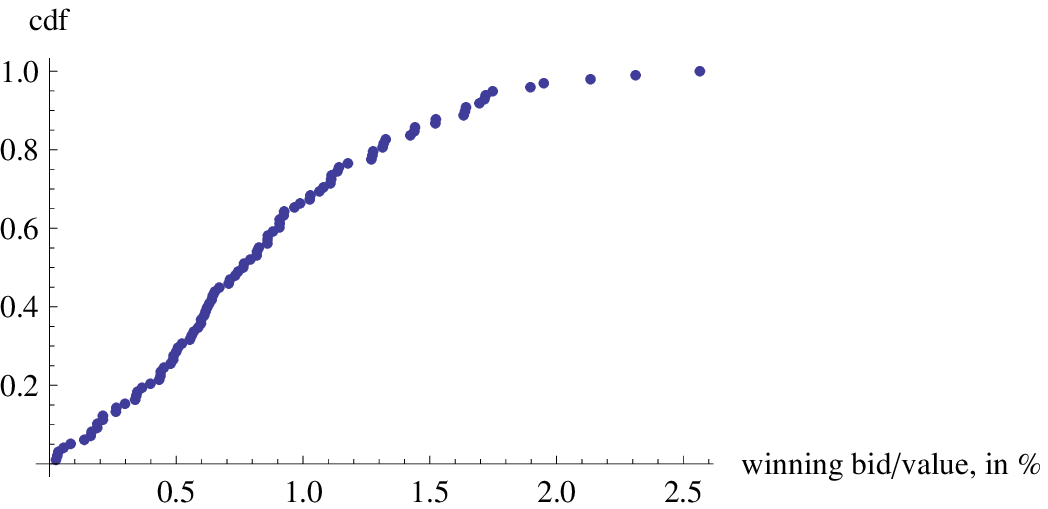}
\caption{The CDF of the ratio winning bid/valuation}
\label{fi:CDF}
\end{figure}

The
winning bid in more than 90\% of the auctions is less than 2\% of the sold
object's value. This suggests that if a bidder bids on {\em all}
numbers up to 2\% of the object's value, he will win the
object with high probability. Because each bid cost 0.5 Euro, the
cost of such a strategy is equal to the object's value. A natural
question that arises now is whether one can shrink the size of the
block on which one bids, thereby lowering the cost of bidding,
without affecting much the probability of winning.

For every $0 \leq x < y \leq 100$ denote by $\sigma_{x,y}$ the
strategy that bids on all numbers between $x\%$ and $y\%$ of the
item's value. When $c=0.5$, the cost of this strategy is $(y-x)v/2$. We would
like to test whether the strategy $\sigma_{x,y}$ can be
profitable, for various values of $x$ and $y$. To this end, we
check how well this strategy fares given our data. This is
equivalent to having one additional bidder in the auction, who
bids at the very last moment by using the strategy $\sigma_{x,y}$.
The fact that this bidder bids at the last moment implies that he
does not affect the way other participants bid. The table in
Figure~\ref{fi:Performance} describes how the strategy $\sigma_{x,y}$ would have
fared in our data set, for various values of $x$ and $y$ (both the
percentage of bids that it would win, and the average amount of
money it would make).

\bigskip

\begin{center}
FIGURE~\ref{fi:Performance} ABOUT HERE
\end{center}¥

\bigskip

\begin{figure}[b]
\includegraphics[width=\textwidth]{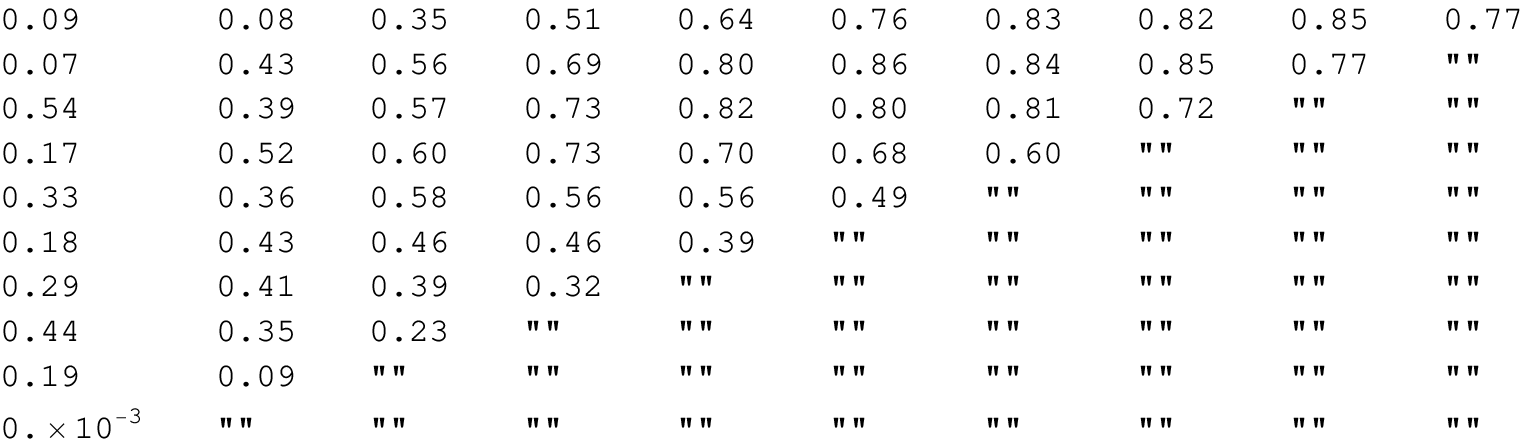}
\caption{Performance of the strategy  $\sigma_{x,y}$, as a function of $x$ and $y$}
\label{fi:Performance}
\end{figure}

This figure should be read as follows. The performance of the strategy $\sigma_{x,y}$ is defined as the ratio between the total value of the items won when using $\sigma_{x,y}$, and the total amount bid. Hence, the higher the ratio, the better the performance. Note also that ratios below one indicate that the gains do not offset the bidding costs, and therefore, losses are incurred.

The table above gives the value of the performance measure, for various choices of $x$ and $y$. Each of the ten rows corresponds to a given value for $x$, ranging from 0.1 to 2\%. On the row $x$, the different numbers correspond to different values for $y$, ranging from $x+0.2$ to $2.1$\%.

It is plain that none of the strategies $\sigma_{x,y}$ achieves a performance of at least one. That is, no
choice of $(x,y)$ would have allowed the additional bidder to claim gains. Optimal values for $(x,y)$ yield a performance of roughly 0.85, that is, a loss of the order of 15\% of the total amount of the bids.

\bigskip
These observations raise a number of remarks.

Note first that the strategy $\sigma_{x,y}$ does not make use of the
dynamic nature of the auction; it is probable that there are
better strategies that do use the dynamic nature of the auction
and make money, as we discuss more in detail below. However, because the internet sites that run
LUBA do not publish dynamic data, we were unable to pursue this
direction.

Note also that the strategy $\sigma_{x,y}$ treats auctions on items with different values in a neutral way.
One might expect that auctions on high-value items do not attract the same bidders, and the same number of bids, as auctions on low-value items. Hence, an obvious refinement of this exercise would consist in adjusting the values of $x$ and $y$ to the item value.

Our numerical analysis raises the following standard statistical concern. While we compute \textit{ex post} the optimal choices for $x$ and $y$, we did not check for the robustness of these estimates. In particular, it may \textit{a priori} be the case that using a different sample of auction data would lead to very different values for $x$ and $y$. As a consequence, relying on one sample to choose $x$ and $y$, and using these estimates to bid in other auctions may lead to a poor perfomance. While we acknowledge this, we observe that the performance of the strategy $\sigma_{x,y}$ does not appear to be sensitive, in a neighborhood of the optimal $(x,y)$. This suggests that the risk is limited in practice.

We believe however that our exercise raises a more subtle issue. As discussed above (see Figure 3),
the winning bid in all 98 auctions but 3 did not exceed 2\% of the item value. Since bidding \textit{all} integer up to 2\% of the value costs exactly the item value, we might expect the bidding strategy $\sigma_{0,2}$ to achieve a performance close to 1. Yet, as Figure 4 shows, our measured performance for this strategy is only 0.77 (last entry of the first row). Inspection of the data shows that this puzzling difference is only to a minor extent due to the fact that the three auctions in which the winning bid exceeded 2\% of the item value featured rather high-value items (valued at 600, 650 and 1149 euros).\footnote{Plainly, if the strategy $x-y$ wins $p$ \% of the auctions, but mainly those on low value items, its performance is below $p$.}
Instead, the discrepancy between the anticipated and the actual performance of the strategy $\sigma_{0,2}$  is to be traced to the following fact. Even though the winning bid turned out not to exceed 2\% in 95 out of 98 auctions, adding  bids according to $\sigma_{0,2}$ would have allowed to win only 84 of these 98 auctions.
As a striking illustration, the auction \#6787 (item value 1049 euros) was won with a bid of 61 cents. Yet, the lowest integer that was \textit{not} offered was 2372 cents -- more than 2\% of the item value, so that playing $\sigma_{0,2}$ at the last instant would have failed to win this auction. In fact, a huge fraction of the integers between 61 and 2372 was offered by only \textit{one} bidder. To be specific, none of these integers was offered by more than 6 bidders, two were offered by 6 bidders, and five by 5 bidders. Each of all remaining 2311 integers was  offered by at most 4 bidders. This provides further evidence of the importance of block bidding. This may suggest that the opportunities for profit are already well-exploited by existing strategic bidders. This also suggests that entering the auctions with the strategy $\sigma_{x,y}$ would have an impact on other bidders' behaviors.

This result suggests that even in practice, strategic bidders may be able to make money in LUBA.
Note that the fact that we added this strategic bidder does not harm the seller,
because the number of bids only increases, so that both the seller and the strategic bidder make money.

\subsection{The winner's strategy}\label{suse:winner}

In the previous subsection we considered a possibly interesting static strategy. Here we discuss the actual strategies that the winners of the various auctions play. In particular we examine a strategy that has been proposed in some blogs, and we argue that it is impossible to use the limited available data to infer the parameters used by the winner in this strategy.

The nature of available data makes it difficult to identify the strategy used by the various bidders, and especially by the winner, but nevertheless shows bids on  large and small intervals of numbers. Our results show that a static strategy, superimposed to the existing bids, in general loses money. In principle this does not refute the possibility that a static strategy could have been adopted by the winner. It just makes it implausible. In any case, based on the existing data,  we would not be able to distinguish a purely static interval strategy from a dynamic one, as described in some blogs. 

A possible dynamic strategy for a rational bidder $i$, say, would involve bidding on some numbers at random in the initial stages of the auction. Towards the end of the auction bidder $i$ picks the lowest unique among her bids (call it $y$), and bids  on a large enough interval of numbers between some value $x < y$ and $y-1$.  The length of the interval should be large enough as to kill all possible lower unique bids of other participants, but small enough as to permit a gain. The bidder can calibrate the lower extreme $x$ of this interval using data from previous similar auctions as follows: call $z$ the largest number such that all numbers below $z$ have been bid by at least three bidders. Look at the distribution of $z$ in the past auctions. The value $x$ could be chosen to be a quantile of the distribution of $z$, for instance the 95\% quantile. This would give a probability of 95\% of winning the auction if $y$ is not killed by some other bidder. This is because with probability .95 all the numbers below $z$ will have at least two bids (plus the possible bid of the winner of the previous auctions), and all the numbers between $x$ and $y-1$ will be killed by the interval bidding. If the interval $[x, y-1]$ is small enough, this will guarantee a gain, if the bidder wins the auction, which happens with probability .95 times the probability that $y$ is not killed.

Once $x$ is well calibrated, this strategy could be very efficient when the bidder who adopts it is the only fully strategic bidder. The other bidders could either be irrational and just bid a few numbers non-strategically or could have some constraints in their strategies, for instance a low budget that does not allow them to bid on large intervals. 

If several bidders act rationally without budget constraints, then the situation is more complicated. If all the rational bidders adopt the same strategy, then they will have the same probability to win the object, therefore the auction will be like a lottery among them. In this case, if a strategic bidder uses the above strategy, then she should take into account the fact that her unique bid  $y$ may be killed by some other bidder at the end of the auction. Therefore she may decide to look for additional unique bids above $y$  and then to kill additional intervals also above $y$.

Even if we assume that various bidders adopt an adjusted  dynamic interval strategy, as described above, different factors may intervene in the choice of its parameters. For instance, bidders may have different risk postures. Some of them may go after some almost sure long term gain (even if small). Some others may choose to seek only reasonably large sporadic gains, even if more uncertain. The risk averse bidders will probably choose to bid on a larger interval of numbers, whereas the risk seekers may want to bid on a smaller interval below a larger unique bid, in the hope that the smaller numbers be mutually killed by other bidders.

One more reason for adopting different dynamic interval strategies could be that bidders have different complexity bounds. For instance they may have  different hardware and software, different historical datasets and especially different computational skills. This could give the bidders different capabilities of calibration of the optimal interval.  If the rational bidders adopt strategies of different complexity, the bidder with the largest complexity will have the highest chance to win the object. 

When all is said and done, recovering the winner's strategy from the data we have does not seem feasible. Even assuming that the winner adopted (a variation of) an interval strategy, the extremes of this interval are difficult to figure out, not to mention the random numbers bid in the first part of the auction. This entails that no reasonably reliable estimate of the winner's gain  can be proposed. Therefore it is also difficult to imagine what fraction of auctions the rational winners actually win and then how many of them obtain a positive gain in the long run.

\section{Concluding comments}
\label{se:comments}

Our study raises several questions.

We identified the unique symmetric equilibrium when there are two bidders,
and proved that in this case all equilibria are in monotone strategies.
\citet{EichVin:mimeo2008} identified all pure equilibria (not necessarily symmetric) of this game, when only monotone strategies are used.
It would of theoretical interest to find all the equilibria of the game without this restriction on the feasible strategies.

To be able to gain more insight on the dynamic behavior of the bidders,
it is crucial to find a tractable model that will include
the particular form of information revelation that this auction method exhibits.
What is the proper model to use? Moreover we would need dynamic data to say something significant on the strategy used by the bidders, on the number of strategic bidders, and on their average gain.

One of our assumptions was that the value of the object, $v$,
is common to the seller and to the bidders.
The analysis can be carried out when the value of the object to the seller is lower than its value to the bidders,
as is usually the case.
It is interesting to know how the results change when the bidders have different private values and when the bidders derive utility from the mere participation in  the auction (like a poker player who draws pleasure by just playing).

\bibliographystyle{artbibst}
\bibliography{bibbidplaza}

\end{document}